\documentclass[11pt,reqno]{article} 
\usepackage{amstext,amsmath,amssymb,amsthm}

\usepackage{latexsym}
\usepackage{exscale}
\usepackage{color}

\newcommand{\R}{{ \mathbb  R  }}

\newcommand{\Z}{  \mathbb Z }
\newcommand{\N}{  \mathbb N }

\renewcommand{\Re}{\mbox{Re\,}}

\newtheorem{Thm}{Theorem}[section]
\newtheorem{Lemma}[Thm]{Lemma}

\newtheorem{Prop}[Thm]{Proposition}
\newcommand{\dsize}{\displaystyle}
\numberwithin{equation}{section}
\newcommand{\CC} {\mathcal C}

\theoremstyle{remark}

\begin{document}
	\begin{center}
		\vskip 1cm{\Large\bf 
			Approximating the divisor functions 
		}
		\vskip 1cm 
		\large
	
		 Andrew Echezabal\footnote{{\it Address}		Andrew Echezabal, Florida International university, department of mathematics and statistics, Miami, FL 33199
		 	andrewechezabal@gmail.com}-
			Laura De Carli\footnote{{\it Address} 
			Laura De Carli, Florida International university, department of mathematics and statistics, Miami, FL 33199
			Orcid:  0000-0003-4420-6433, decarlil@fiu.edu}
		-
		Maurizio Laporta\footnote{$^*$Corresponding author. M.Laporta,
			Universit\`a degli Studi di Napoli,
			Dipartimento di Matematica e Applicazioni,
			Complesso di Monte S. Angelo, Via Cinthia,
			80126 Napoli (NA), Italy.\par
			Orcid: 0000-0001-5091-8491. mlaporta@unina.it}$^*$
	\end{center}

	\begin{abstract} We  exploit the properties of  a sequence of functions  that approximate the divisor functions and combine them with an analytical formula  of a delta-like sequence to  give a new  proof of a theorem of Gr\"onwall on the asymptotic of the divisor functions. 
		
\end{abstract}

	\section{Introduction and statement of the main results}\label{intro}

	 Given a complex number $\alpha$, for any $m\in\N   =\{ 1,2,\ldots\}$ 
	 the sum of positive divisors function $\sigma_{\alpha}(m)$ is defined as the sum of the $\alpha$th powers of the positive divisors of $m$ \cite{HW}. It is usually expressed as
	 $$
	 \sigma_{\alpha}(m)=\sum_{k\mid m }k^{\alpha}, 
	 $$
	 where $k|m$ means that $k\in\N$ is a divisor of $m$. 
	 The functions so defined are called 
	 {\it divisor functions}. They play an important role in  number theory, namely within the study of divisibility properties of integers and the distribution of prime numbers. Such functions appear in connection with the Riemann zeta function $$
	 \zeta(s)=\sum_{k=1}^\infty \frac{1}{k^s},\quad (\Re s>1).
	 $$ 
	 Indeed, given any $\alpha\in\R$, it is plain that
	 $$
	 \sigma_\alpha(m)=\sum_{k\mid m }\left(\frac{m}{k}\right)^{\alpha}\le m^{\alpha}\sum_{k=1}^m \frac{1}{k^\alpha}.
	 $$
	 In particular, for $\alpha>1$ this yields
	$$
	G_\alpha(m):=\frac{\sigma_{\alpha}(m)}{m^\alpha}\le 
	 \zeta(\alpha),\ \forall m\in\N.
	 $$
	In 1913, Gr\"onwall \cite{Gr} showed in a quite elementary way that if $\alpha>1$, then
	 \begin{equation}\label{Gronwall}
	 \limsup_{m\to\infty}G_\alpha(m)= \zeta(\alpha).
	 \end{equation}
 In fact, in view of the previous inequality, immediately he derives \eqref{Gronwall} after proving that 
 $$
 \lim_{m\to\infty}G_\alpha(P_m)= \zeta(\alpha),
 $$
 where
 $$
 P_m:=\prod_{i=1}^mp_i^m
 $$
 is the $m$th power of the product of the first $m$ prime numbers $p_1=2,p_2=3,p_3=5, \ldots,p_m$.
 
 By arguing in an alternative way, in \S\ref{proof} we prove the following result, which   still yields Gr\"onwall's equation \eqref{Gronwall}.
 \begin{Thm}\label{T-Est-divisors}
 	For every  real number $\alpha>1$  one has
 	$$
 	\lim_{m\to\infty}G_\alpha(m!)=\lim_{m\to\infty}G_\alpha\left([1,2,\ldots,m]\right)= \zeta(\alpha),
 	$$
 	where $[1,2,\ldots,m]$ is the least common multiple of $1,2,\ldots,m$.
 \end{Thm}

The proof of this theorem is arguably  more involved than Gr\"onwall's one.
However, a novelty of our approach lies in the fact that, unlike Gr\"onwall, we do not use the multiplicativity of $\sigma_{\alpha}$, i.e., $\sigma_{\alpha}(mn)=\sigma_{\alpha}(m)\sigma_{\alpha}(n)$ for any coprime $m,n$ \cite{HW}, but we  exploit the properties of a sequence of  $C^\infty$ functions whose pointwise limit is $\sigma_{\alpha}$ (see \S\ref{ADF}).    

Our main tool is the following theorem, which, to the best of our knowledge, is new in the literature. 
\begin{Thm}\label{T-approx-cos} 
	Let $\eta_1, \eta_2\in\R\setminus\Z$ or $\eta_1, \eta_2\in\Z$, with $\eta_1 \leq \eta_2$.
	For every   $f \in L^1[\eta_1, \eta_2 ] $ which   is  continuous in a neighborhood of the integers,  we have 
	\begin{equation}\label{e11general}
		\lim_n\sqrt{n\pi} \int_{ \eta_1 }^{ \eta_2 } \cos^{2n}(\pi x) f(x){\rm d}x=
		\dsize\sum_{m=\lfloor \eta_1\rfloor}^{\lfloor \eta_2\rfloor}  f(m)     -
		\frac{f( \eta_1)\chi_{_\Z}( \eta_1)+f( \eta_2)\chi_{_\Z}( \eta_2)}{2},
	\end{equation} 
where  $\lfloor \eta\rfloor $ denotes the integer part of $\eta\in\R$ and $\chi_{_\Z}$ is the characteristic function of $\Z$.
\end{Thm}

 Our hope is also that this alternative approach might provide   new insights and lead to some new result in the future.

 \medskip\noindent
  {\bf Acknowledgements.}
 The  approximation of  the divisor functions with   the sequences of $C^\infty$ functions presented in this paper   is  an original idea  of the first Author of this paper and the core of his Master thesis at Florida International University.

%\section{Preliminaries}

\subsection{Notation} For brevity, we write 
$\displaystyle\lim_n$ instead of  $\displaystyle\lim_{n\in\N\atop n\to\infty}$, and the symbol $\{\cdot\}_{n\in\N}$ for a sequence  is abbreviated as $\{\cdot\}_n$.
The power $\left(\cos(y)\right)^\beta$ is written as
$\cos^\beta(y)$, and the same we do for the sine function. 

Given a real number $x\ge 1$,   in sums like $\displaystyle \sum_{k=1}^{x}$ we mean that $k\in\N$ with $k\le \lfloor x\rfloor$, where $\lfloor x\rfloor $ denotes 
the integer part of $x\in\R$, i.e.,
the greatest integer such that $x- \lfloor x\rfloor \ge 0$. 

The characteristic function of a set $A\subseteq \R$ is denoted by $\chi_{_A}$.

\section{Approximation of the divisor function}\label{ADF}

In this section a sequence of   functions that approximate  $  \sigma_{\alpha}$  %and $\sigma_{\alpha}^*(x)$ 
is defined.  To this end,
we consider $\sigma_{\alpha}$   for real    $\alpha>0$, and 
we extend such a function to   $[0, \infty)$ by letting $\sigma_{\alpha} (x)=0$ when $x\in [0, \infty)\setminus\N$.

For any given real numbers $\alpha> 0$ and $M>1$,  let us consider the sequence $\{\CC_{\alpha,  n}\}_n\subset C^\infty[0, M]$  
defined as

\begin{equation}\label{def-cos2}   
	 \CC_{\alpha,n}(M;x) =  \sum_{ k=0} ^M f_n(k),
 \end{equation}   
where
$$
f_n(t)=f_{\alpha,n}(x;t):=\begin{cases} 0   & \mbox{ if $t=0$ or $x=0$,}\cr t^\alpha  \cos^{2n}\big(\frac{\pi x}{t} \big)
	& \mbox{ otherwise.} \end{cases}
$$

The main property of the functions  $\eqref{def-cos2}$  is  given by the following proposition.

\begin{Prop}\label{T-cos-approx} For any given real numbers $\alpha> 0$, $M>1$
	and  $x\in [0, M]$,     the   sequence  $n\to \CC_{\alpha,  n}(M;x)$   is  non-increasing,  and  $\dsize \lim_n \CC_{\alpha,  n}(M;x)=  \sigma_\alpha(x) .$ 
\end{Prop}

\begin{proof} If $x=0$, then this is obviously true. Let $x\in (0,M]$ and $k \in (0, M]\cap\N$  be fixed.
	If $x\not\in\N$ or if $x\in\N$ is not divisible by  $k$,  then  $|\cos(\pi x/k) |<1$.
	In this case,  the sequence $n\to \cos^{2n }(\pi x/k)$  is decreasing  and $\dsize\lim_n\cos^{2n } (\pi x/k)=0$.
	On the other hand,
	$\cos^{2n }(\pi x/k)=1$ for all $n\in\N$  if and only if  $x\in\N$ and $k$ divides $x$. Consequently, 
	$$
	\lim_n \CC_{\alpha,  n}(M;x)=
	\begin{cases} 
		\dsize\sum_{k\mid x }k^{\alpha} &  \mbox{if \ $x\in(0,M]\cap\N$,}\cr
		0& \mbox{ otherwise  },	 
	\end{cases}
	$$	 
	as claimed.
\end{proof}

  \section{Delta-like sequences   and proof of Theorem \ref{T-approx-cos}}

In the literature (see e.g. \cite{R}, \cite{B})
a sequence of non-negative functions $\{g_n\} _{n\in\N}\subset L^1(\R^m)$  is known to be a
{\it delta-like sequence} if
\begin{equation}\label{approxid}  
\lim_{n }\int_{\R^m} g_n(x)f(x){\rm d}x=f(0)
\end{equation}  
for all $f\in C^\infty_0(\R^m)$.
If, in addition, the functions $g_n$ are smooth, they are also called {\it mollifiers}, or {\it approximations of the identity}.

A well-known set of delta-like sequences is made of   non-negative functions  $g_n\in L^1(\R^m)$   satisfying  the properties (a) and (b) below.

\smallskip
(a) For every $\epsilon>0$, there exists  $\delta >0$ such that
$\dsize\int_{|x|>\delta}g_n(x){\rm d}x <\epsilon$.     
  
 (b)  
 $\dsize\int_{\R^m} g_n(x){\rm d}x=1$.
 
 \medskip
Here and in what follows, we have let  $|x|=\sqrt{x_1^2+... +x_m^2}$.
  
  We prove the following 
 \begin{Prop} \label{P-delta-like}
 If $\{g_n\} _{n\in\N}\subset L^1(\R^m)$, where each $g_n$ is non-negative, satisfies the previous property (a), and 
  \smallskip
  
 (b') $\dsize\lim_n  \int_{\R^m} g_n(x){\rm d}x=1$,
   
 then
 \eqref{approxid} holds for every $f\in L^1(\R^m)\cap L^\infty(\R^m)$ which is continuous in a neighborhood of $x=0$.
 	\end{Prop}

 \begin{proof}  From (b') it follows that for every fixed $0 <\epsilon<1$ there exists $N>0$ such that  
 	$\dsize 1-\epsilon< \int_{\R^m} g_n(x){\rm d}x < 1+\epsilon $ whenever $n\ge N$. The continuity of $f$  at $x=0$ yields    $|f(x)- f(0)|<\epsilon$ if $|x|<\delta$ 
 	for some  $\delta>0$.  
 Thus, by using (a) we see that  if $n\ge N$, then
\begin{align*}
  \int_{\R^m} g_n(x)\big(f(x)- f(0)\big){\rm d}x  &=  \int_{|x|<\delta}  g_n(x)\big(f(x)- f(0)\big){\rm d}x+\int_{|x|>\delta}  g_n(x)\big(f(x)- f(0)\big){\rm d}x 
  \\ & <    \epsilon \int_{|x|<\delta}  g_n(x){\rm d}x + 2||f||_\infty \int_{|x|>\delta}  g_n(x){\rm d}x
   \\ & <   \epsilon  (1+\epsilon) + 2||f||_\infty \epsilon.
  \end{align*}
We also  have that
\begin{align*}
	\int_{\R^m} g_n(x)\big(f(x)- f(0)\big){\rm d}x  & \ge  -\int_{|x|<\delta}  g_n(x) | f(x)- f(0)|{\rm d}x - \int_{|x|<\delta}  g_n(x) | f(x)- f(0)|{\rm d}x
	\\ & >   -\epsilon \int_{|x|<\delta}  g_n(x){\rm d}x  -2||f||_\infty\int_{|x|>\delta}  g_n(x){\rm d}x
	\\ & >-  \epsilon  (1+\epsilon) -2||f||_\infty\epsilon.  
\end{align*}

   Since $\epsilon$ is arbitrary,  it follows that \begin{align*}
   0=\lim_{n }\int_{\R}   g_n(x) ( f(x)-f(0)){\rm d}x &=\lim_{n }\int_{\R}  g_n(x) f(x)  {\rm d}x-f(0)\lim_{n }\int_{\R}g_n(x) {\rm d}x\\ &= \lim_{n }\int_{\R}  g_n(x) f(x)  {\rm d}x-f(0),
\end{align*}
which yields \eqref{approxid}.
\end{proof}

\subsection{ Proof of Theorem \ref{T-approx-cos}}
 We need the following 
    
\begin{Lemma} \label{L-delta-seq} Let $ \varphi_n(x):= \sqrt{n\pi}\cos^{2n}(\pi x) 
	$ and $\eta\in(0,1)$.
	If $f\in L^1 (-\eta, 1-\eta )$ is continuous at $x=0$, then
	\begin{equation}\label{e-limit-1}
		\lim_{n} \int_{-\eta}^{1- \eta}f(x) \varphi_n(x){\rm d}x= f(0).
	\end{equation} 
	If $f\in L^1 [0,1]$ is continuous at $x\in\{0,1\}$, then
	\begin{equation}\label{e-limit-2}
		\lim_{n} \int_{0}^{1 }f(x) \varphi_n(x){\rm d}x= \frac{f(0)+f(1)}{2}.
	\end{equation}
\end{Lemma}
 
 \begin{proof}
 We first show that the  function $\varphi_n  
 	\chi_{_{(-\eta, 1-\eta )}}(x)$  satisfies the properties (a) and (b')  in Proposition \ref{P-delta-like}. 
 %	If $f\in L^1 [0,1]$ is continuous at $x=0$  and $x=1$, we have that
 %	\begin{equation}\label{e-limit-2}
 	%	\lim_{n} \int_{0}^{1 }f(x) \varphi_n(x){\rm d}x= \frac{f(0)+f(1)}{2}.
 %	\end{equation}
 	%
Without loss of generality we   assume   $\eta =1/2$ and for simplicity the function $\varphi_n  
\chi_{_{(-1/2, 1/2 )}}$ will still be denoted by $\varphi_n$.
  	
 In order to prove  (a), we show that the sequence $\{\varphi_n\}_n$ converges uniformly to zero in any compact subset of $\R\setminus\{0\}$, i.e.,    $\lim_{n}\sup_{|x| >\delta}\varphi_n(x)=0 $
 	for every $ \delta >0$.
 	Since the $\varphi_n$ are even  and  decreasing when $x>0$, it suffices to show that  $\lim_{n}\varphi_n(\delta)=0$ for every $0<\delta <1/2$. 
 	By    the elementary   inequality $ \sin( t) \ge  2t/\pi  $  for all $t \in [0, \,\pi/2]$, we have 
 	$$
 	\lim_{n}\varphi_n(\delta)=\lim_{n}\sqrt{n\pi}  \cos^{2n}(\pi \delta)  = 
 	\lim_{n}   \sqrt{n\pi}\big(1-\sin^2(\pi \delta)\big)^n \leq \lim_{n} \sqrt{n\pi}(1- 4 \delta^2)^{n}  =0. 
 	$$
 	Hence, (a) is proved.
 	
 \medskip	
 We now prove (b').
  To this end, first let us
recall the beta function \cite[Formulae  8.380\, 2, 8.384\, 1]{GR}
 	$$
 	\beta(a,b)= \frac{\Gamma(a)\Gamma(b )}{\Gamma(a+b)}:=2\int_{0}^{\frac \pi2} \sin^{2a-1}(t)\cos^{2b-1}(t){\rm d}t,\quad \Re{a},\Re{b}\in(0,+\infty),
 	$$  
 	where $\Gamma(z);=\int_0^\infty e^{-t}t^{z-1}{\rm d}t$, $\Re{z}>0$, is the well-known gamma function \cite[Formula 8.310\, 1]{GR}. Since $\Gamma(1/2)=\sqrt\pi$ \cite[Formula 8.338\, 2]{GR},
 	we can write 
 	\begin{align*}
 		\int_{-\frac 12}^{\frac 12} \varphi_n(x){\rm d}x=&\ 2\,\sqrt{n\pi} \int_{0}^{\frac 12 } \cos^{2n}(\pi x){\rm d}x=   2\,\sqrt{\frac      n\pi} \int_{0}^{\frac \pi 2} \cos^{2n}(t){\rm d}t\nonumber\\
 		=&\
 		\sqrt{\frac      n\pi}\, \beta\mbox{$(\frac 12 ,\, n+\frac 12)$}=
 		\frac{   \Gamma \left(n+\frac{1}{2}\right)}{   \Gamma
 			(n+1)}\sqrt{n}. \label{e-int-cos-n}
 	\end{align*} 
 	Hence, (b')  follows after observing that
 	by Stirling's approximation formula we get
 	\cite[Formulae 8.339\, 1, 2]{GR}
 	$$
 	\lim_n \frac{ \Gamma(n+\frac 12)}{ \Gamma(n+1)}\sqrt n=\sqrt\pi\lim_n \frac{ (2n-1)!!}{ n!2^n}\sqrt n=
 	1.
 	$$ 
 	Now, since   $f\in L^1[-\frac 12,\frac 12]$  is continuous at $x=0$, from Proposition \ref{P-delta-like} we conclude that 
 	\begin{equation}\label{e-lim} \lim_n  \int_{-\frac 12 }^{\frac 12 } \varphi_n(x)  f(x){\rm d}x= \lim_n\sqrt{n\pi} \int_{-\frac 12 }^{\frac 12 } \cos^{2n}(\pi x) f(x){\rm d}x=f(0).\end{equation}     
  Thus, \eqref{e-limit-1} is proved. 
  
  \medskip
  It remains to prove \eqref{e-limit-2}.
  Since $f\in L^1[0,1]$ is continuous at $x\in\{0,1\}$,
 	the functions  $x\to f(|x|)$ and  $x\to f(1-|x|)$  are continuous at  $x=0$. 
 Further,
 	$\varphi_n(x)= \varphi_n(1-x)$   and $\varphi_n(x)= \varphi_n(-x)$. Thus, we  can write
 	\begin{align*}
 		\int_{0}^{1 }f(x) \varphi_n(x){\rm d}x =& \int_{0}^{\frac 12} f(x)  \varphi_n(x){\rm d}x\, +  
 		\int_ {\frac 12}^1 f(x) \varphi_n(x){\rm d}x
 		\\
 		=& \int_{0}^{\frac 12} f(x)  \varphi_n(x){\rm d}x\, +  
 		\int_ {0}^{\frac 12} f(1-x) \varphi_n(x){\rm d}x
 		\\
 		=&
 		\int_{-\frac 12}^{0} f(-x)  \varphi_n(x){\rm d}x\, +  
 		\int_ {-\frac 12}^0 f(1+x) \varphi_n(x){\rm d}x
 	\end{align*}
 	From the above chain of identities  it follows that 
 	$$2\int_{0}^{1 }f(x) \varphi_n(x){\rm d}x=
 	\int_{-\frac 12}^{\frac 12} f(|x|)  \varphi_n(x){\rm d}x\, +  
 	\int_ {-\frac 12}^{\frac 12} f(1-|x|) \varphi_n(x){\rm d}x.
 	$$
 	Hence,  \eqref{e-lim} yields    
 	$\dsize 2\lim_n \int_{0}^{1 }f(x) \varphi_n(x){\rm d}x =f(0)+f(1),
 	$ 
 as required.
 \end{proof}
 
 \begin{proof}[Proof of  Theorem \ref{T-approx-cos}]  
 	We can assume   that $ [\eta_1, \eta_2]= [a-\eta, b- \eta]$, where $a\leq b$ are integers and $0\le \eta<1$. This is obvious if $\eta_1, \eta_2\in\Z$. Therefore, let us assume that $\eta_1, \eta_2\not\in\Z$, so that $\eta:= \lfloor  \eta_1+1 \rfloor-\eta_1\in(0,1)$.
 	If we let  $a= \lfloor  \eta_1+1 \rfloor$ and $b=\lfloor  \eta_2+2 \rfloor $, then $\eta_1= a-\eta$ and  $b-\eta=
 	\lfloor  \eta_2+2 \rfloor-\eta>\eta_2+1 -\eta>\eta_2$, i.e., $\eta_2=b-\eta-\eta'$ for some $\eta'>0$.
 	Consequently, 
 	 we can extend  $f\in L^1[\eta_1, \eta_2]=L^1[a-\eta, b-\eta-\eta']$  to $\tilde f\in L^1[a-\eta,  b-\eta ]$ by letting $\tilde f\equiv 0$  in $[b-\eta-\eta', \ b-\eta]$. Note that 
$\tilde f$ is  still continuous  in a neighborhood of the integers. Further,  
the sum  on the right-hand side of \eqref{e11general} does not change if $f$ is replaced by $\tilde f$. Thus, without loss of generality, we can assume that $f$  is defined in  the interval $[a-\eta,\  b-\eta]$ and write 
 \begin{align}\nonumber  \lim_n\sqrt{n\pi} \int_{a-\eta }^{b-\eta } \cos^{2n}(\pi x) f(x){\rm d}x &= \sum_{m=a}^{b-1} \lim_n\sqrt{n\pi}\int_{m-\eta }^{m+1-\eta } \cos^{2n}(\pi x) f(x){\rm d}x 
 	\\\label{arg1} &=
 	\sum_{m=a}^{b-1} \lim_n \int_{ -\eta}^{  1-\eta } \varphi_n(x)  f(m+x){\rm d}x,
 \end{align}
 where $\varphi_n$ has been introduced in Lemma \ref{L-delta-seq}. Indeed, by applying this lemma we have 
 $$
 \lim_n \int_{ -\eta}^{  1-\eta } \varphi_n(x)  f(m+x)=\begin{cases} f(m)   & \mbox{ if $\eta>0$,}\cr\cr \dsize\frac{f(m)+f(m+1)}{2}
	& \mbox{ if $\eta=0$.} \end{cases}
$$  
Consequently,    \eqref{arg1} becomes \eqref{e11general}.
\end{proof}
 
\section{Proof of Theorem \ref{T-Est-divisors}}\label{proof}

For $x,n\in\N$, with $x>1$, let us take $M=2x$ for the approximants  of $\sigma_\alpha(x)$  
	defined in  \eqref{def-cos2}, namely 
	$$
	\CC_{  \alpha, n}(2x;x)=\sum_{k=0}^{  2x} k^\alpha  \cos^{2n} \big(\frac {\pi x} {k} \big). 
	$$   
	 Recall that the function  $f_n(t):=  t^\alpha\cos^{2n}(\pi   x/ t)$ extends to  a continuous   function  in $ [0, \infty)$, with $f_n(0)=f_n(2x)=0$. Further, it satisfies the hypotheses of  Theorem \ref{T-approx-cos} in $[0,2x]$. Thus,
	\begin{align*} 
	\CC_{  \alpha, n}(2x;x)&=  \lim_N\ \sqrt {N\pi}\int_{0}^{  2x   }  f_n(t) \cos^{2N}(\pi t){\rm d}t +\frac{f_n(0)+f_n(2x)}{2}\\
	&=\lim_N\ \sqrt {N\pi}\int_{0}^{  2x   } f_n(t) \cos^{2N}(\pi t){\rm d}t.
\end{align*} 
	By the  change of variables $t\to    xt$ one has
	$$  
	\CC_{  \alpha, n}(2x;x) = x\lim_N  \sqrt {N\pi}\,  F_{n,N}(x),  
	$$
	where 
	$$ 
	F_{n,N}(x):= \int_{0}^{2 }  \!f_n(xt)\cos^{2N}(  \pi xt){\rm d t},    
	$$
	$$
	f_n(xt)=\begin{cases} 0   & \mbox{ if $t=0$}\cr (xt)^\alpha  \cos^{2n}\big(\frac \pi {  t} \big)
		& \mbox{ otherwise.} \end{cases}
	$$
	Thus, Theorem \ref{T-cos-approx} yields
	\begin{equation}\label{est-sigma-a1}\sigma_\alpha(x)=x\lim_n\lim_N \sqrt {N\pi}\, F_{n,N}(x)   ,\quad \forall x\in\N.  
	\end{equation}
Since it is plain that $f_n(xt)\ge 0$ for all $t\in[0,2]$, we
	can take a sufficiently large integer $m$, set $m'=4(m!)$, and write
	\begin{align*}
		F_{n,N}(x)  &\ge   \sum_{k=1}^{m}  \int_{\frac{1}{k}-\frac{1}{m'}}^{\frac{1}{k}+\frac{1}{m'}}  \!\!f_n(xt)\cos^{2N}(  \pi xt){\rm d}t\\	&\ge  \, x^\alpha\sum_{k=1}^{m}  \big(\frac{1}{k}-\frac{1}{m'}\big) ^\alpha   \int_{\frac{1}{k}-\frac{1}{m'}}^{\frac{1}{k}+\frac{1}{m'}}  \! 
		\cos^{2n}\big(\frac \pi {  t} \big) \cos^{2N}(  \pi xt){\rm d}t
		\\
		&=  x^{\alpha-1}\, \sum_{k=1}^{m} \big(\frac{1}{k}-\frac{1}{m'}\big)^\alpha  \int_{\frac{x}{k}-\frac{x}{m'}}^{\frac{x}{k}+\frac{x}{m'}} \cos^{2n}\big(\frac { \pi x} { t} \big)\cos^{2N}(  \pi t){\rm d t},
\end{align*}
after the change of variable $xt \to t$, 
Now, let us take    $x=m! $  so that $x/k=m!/k\in\N$ and $\frac{x}{k}\pm\frac{x}{m'}=\frac{m!}{k}\pm\frac{1}{4}\not\in\Z$
for  all $k\in\{1,\ldots,m\}$. Thus, the previous inequality becomes
$$
F_{n,N}(m!) \ge (m!)^{\alpha-1}\, \sum_{k=1}^{m} \big(\frac{1}{k}-\frac{1}{m'}\big)^\alpha  \int_{\frac{m!}{k}-\frac{1}{4}}^{\frac{m!}{k}+\frac{1}{4} } \cos^{2n}\big(\frac { \pi m!}{ t} \big)	\cos^{2N}( \pi  t){\rm d}t.
$$
 Again, we apply Theorem  \ref{T-approx-cos} by taking $f(t)= 
\cos^{2n}(\pi m!/t)$  see
 that 
\begin{align*}
\lim_{N} \sqrt {N\pi}\int_{\frac{m!}{k}-\frac{1}{4}}^{\frac{m!}{k}+\frac{1}{4} } \cos^{2n}\big(\frac { \pi m!}{ t} \big)	\cos^{2N}( \pi  t){\rm d}t 
& =f\big(\frac{m!}{k}\big)+f\big(\frac{m!}{k}-1\big)\\
&=
 \cos^{2n}(k\pi)+\cos^{2n}\left(k\pi\frac{m!}{m!-k}\right).
\end{align*}
Since for  any given $k\in\{1,\ldots,m\}$ one has $\cos^{2n}(k\pi)=1$ and
$$
\lim_{n}\cos^{2n}\left(k\pi\frac{m!}{m!-k}\right)=
\lim_{n}\cos^{2n}\left(\pi\frac{k}{m!/k-1}\right)=0,
$$
we infer
\begin{equation*}\lim_{n} \lim_N \sqrt {N\pi}F_{n,N}(m!)\ge (m!)^{\alpha-1}\,  \sum_{k=1}^{m} \big(\frac{1}{k}-\frac{1}{4(m!)}\big)^\alpha,
\end{equation*}
which, together with \eqref{est-sigma-a1}, implies that
 $$
G_\alpha(m!)= \frac{\sigma_\alpha(m!)}{(m!)^\alpha} \ge  \sum_{k=1}^{m}\frac{1}{k^\alpha}+\Delta_\alpha(m),
 $$
 where 
 $$
 \Delta_\alpha(m):=\sum_{k=1}^{m}\left( \big(\frac{1}{k}-\frac{1}{4(m!)}\big)^\alpha-\frac{1}{k^\alpha}\right).
 $$
 Since
 by the mean value theorem we see that
 $$
 |\Delta_\alpha(m)|\le \frac{\alpha/4}{m!}
 \sum_{k=1}^{m}\frac{1}{k^{\alpha-1}}\le \frac{\alpha/4}{(m-1)!},
 $$
we deduce 
$$ 
\lim_m G_\alpha(m!) = \zeta(\alpha),
$$ 
because, as already mentioned in \S\ref{intro}, if $\alpha>1$, then
$$
G_\alpha(m)\le 
\zeta(\alpha),\ \forall m\in\N.
$$
It is plain that in order to prove the same limit for $G_\alpha([1,2,\ldots,m])$ 
it suffices to take $x=[1,2,\ldots,m]$ and 
proceed in a completely analogous way.

Theorem \ref{T-Est-divisors} is completely proved.


\begin{thebibliography}{WW}
 	
 	\bibitem{HW} 
 	Hardy, G.H., and  E.M. Wright. 1979. {\it An introduction to the theory of numbers}, (fifth edition). Oxford: The Clarendon Press Oxford University Press.
 	
 	
 	\bibitem{B} Brezis, H. 2010. {\it Functional Analysis, Sobolev Spaces and Partial Differential Equations}. New York, NY:
 		Springer Science \& Business Media. 
 		
 	\bibitem{GR}  Gradshteyn, I.S., and I.M. Ryzhik.  1994. {\it Tables of integrals, series, and products}, (fifth edition), Jeffrey, Al. (ed.). San Diego, California: Academic Press.
 	
 	\bibitem{Gr} Gr\"onwall, T.H. 1913. Some asymptotic expressions in the theory of numbers. {\it  Transactions of the American Mathematical Society}, 14: 113–122.
 	
 		
 	
 	\bibitem{R} Rudin, W. 1991. {\it Functional analysis}, (second edition). New York, NY:
 		McGraw-Hill, Inc.  
 		
 		
 	
 	
 	
 
 	
 	
 \end{thebibliography}
\end{document}